\documentclass[pdf, 10pt]{amsart}

\usepackage{amsmath,amstext,amssymb,amsopn,amsthm,mathrsfs,wasysym,dsfont,comment,xcolor}
\usepackage[OT1]{fontenc}
\usepackage{comment}
\usepackage{bbm}

\usepackage{tikz}
\usetikzlibrary{arrows}
\usepackage{float}

\allowdisplaybreaks

\newtheorem{thm}{Theorem}[section]
\newtheorem{cor}[thm]{Corollary}

\newtheorem{prop}[thm]{Proposition}

\newtheorem{lem}[thm]{Lemma}

\newtheorem{rem}[thm]{Remark}

\def\R{{\mathbb{R}}}
\def\T{{\mathbb{T}}}

\def\N{{\mathbb{N}}}
\def\Z{{\mathbb{Z}}}

\newcommand{\supp}{\mbox{supp}\,}

\renewcommand{\d}{\delta}

\newcommand{\eps}{\epsilon}

 \numberwithin{equation}{section}

\begin{document}


\thanks{The second author is partially supported by the  NSF grant DMS-1800305.
	The third author is supported by the National Science Centre of Poland within the research project OPUS 2017/27/B/ST1/01623.
}

\begin{abstract}
We prove $\ell^p$-improving estimates for the averaging operator along the discrete paraboloid in the sharp range of $p$ in all dimensions $n\ge 2$.
\end{abstract}

\title[] {Sharp $\ell^p$-improving estimates for the discrete paraboloid}

\author{Shival Dasu}
\address{Department of Mathematics, Indiana University,  Bloomington IN}
\email{sdasu@iu.edu}

\author{Ciprian Demeter}
\address{Department of Mathematics, Indiana University,  Bloomington IN}
\email{demeterc@indiana.edu}

\author{Bartosz Langowski}
\address{Department of Mathematics, Indiana University,  Bloomington IN \newline
Wroc\l{}aw University of Science and Technology,\newline
Faculty of Pure and Applied Mathematics, Wroc\l{}aw, Poland}
\email{balango@iu.edu}

\maketitle

\section{Introduction and notation}
In \cite{HKLMY} the authors study averaging operators along a discrete moment curve. More precisely, they consider
\begin{equation*}
\widetilde{A}_Nf(x_1,\dots, x_n)=\frac{1}{N}\sum_{k=1}^N f(x_1+k, x_2+k^2,\dots, x_n+k^n),\qquad (x_1,\dots, x_n)\in\Z^n,
\end{equation*}
and prove (see \cite[Theorem 1.14]{HKLMY}) for $n\ge 3$ and $2-\frac{2}{n^2+n+1}<p\le 2$ the $\ell^p$-improving estimate
$$
\|\widetilde{A}_N f\|_{\ell^{p'}(\Z^n)}\lesssim N^{-\frac{n(n+1)}{2}(\frac{2}{p}-1)}\|f\|_{\ell^p(\Z^n)}.
$$
The range of $p$ in their theorem is not  sharp. Testing the above estimate with standard examples suggests that the optimal range should be $2-\frac{2}{n^2+n}\le p\le 2$.
\smallskip

In this paper we use the circle method to prove the optimal bounds for the averaging operator along a discrete paraboloid. In particular, our main result, Theorem \ref{thm:main}, gives a sharp estimate (except for an $N^\epsilon$ term at the endpoint) for the averages along the discrete moment curve in dimension $n=2$ (in that case the moment curve and the parabola coincide).
\smallskip

We remark that the $\ell^p$-improving estimates in the discrete setting have been studied extensively in the recent years, see e.g. \cite{AM, HKLMY, K, KL, madrid}.

To state our result we need to define first the discrete paraboloid
$$
\mathbb{P}_{n-1}^N=\{(k_1,\dots, k_{n-1}, k_1^2+\dots+k_{n-1}^2)\in\mathbb{Z}^n:1\le k_i \le N, i=1,\dots, n-1\}.
$$
For $f:\mathbb{Z}^n\rightarrow\mathbb{C}$ we consider the averaging operator
\begin{align*}
A^{\mathbb{P}}_N f(x)=\frac{1}{N^{n-1}}\sum_{k_1=1}^N\dots \sum_{k_{n-1}=1}^N f(x_1+k_1, \dots, x_{n-1}+k_{n-1}, x_n+k_1^2+\dots+k_{n-1}^2).
\end{align*}
The main theorem of the paper reads as follows.
\begin{thm}[$\ell^p$ improving for the paraboloid]
\label{thm:main}
Let $\frac{n+3}{n+1}<p\le 2$. The following bound holds
\begin{equation}\label{eq:main}
\| A^\mathbb{P}_N f\|_{\ell^{p'}(\Z^n)}\lesssim  N^{-(n+1)(\frac{2}p-1)}\|f\|_{\ell^p(\Z^n)}.
\end{equation}
Moreover, the above result is essentially sharp in two ways. First, the exponent $-(n+1)(\frac{2}p-1)$ cannot be improved when $p$ is in our range. Second, the above inequality is false if $p\notin [\frac{n+3}{n+1},2]$.
\end{thm}
The operator norm bound in this theorem should be compared with the trivial estimate, for $1\le p\le \infty$
\begin{equation}
\label{trivial}
\| A^\mathbb{P}_N f\|_{\ell^{p}(\Z^n)}\le  \|f\|_{\ell^p(\Z^n)}.\end{equation}
This is due to the uniform integrability of the kernel $\frac1{N^{n-1}}K_N^\mathbb{P}$, with
$$K_N^\mathbb{P}(x)=\sum_{k_1=1}^N\dots \sum_{k_{n-1}=1}^N\ \delta_{(k_1,\dots,k_{n-1}, k_1^2+\dots+k_{n-1}^2)}(x).$$
Interpolating \eqref{eq:main} with the trivial bounds
$$\| A^\mathbb{P}_N f\|_{\ell^{\infty}(\Z^n)}\le \|f\|_{\ell^\infty(\Z^n)}$$
and
$$\| A^\mathbb{P}_N f\|_{\ell^{1}(\Z^n)}\le \|f\|_{\ell^1(\Z^n)}$$
shows that the estimate
$$\| A^\mathbb{P}_N f\|_{\ell^{q}(\Z^n)}\lesssim  N^{-(n+1)(\frac1p-\frac1q)}\|f\|_{\ell^p(\Z^n)}$$
holds for all $(\frac1p,\frac1q)$ inside the triangle in Figure 1.
\medskip

It is also worth observing that $\| A^\mathbb{P}_N \|_{\ell^p(\Z^n)\mapsto\ell^{q}(\Z^n)}=\infty$ if $q<p$. Indeed, if not, then for each $f\in\ell^p(\Z^n)$ and $\textbf{h}\in\Z^n$, writing  $f_\textbf{h}(x)=f(x)+f(x+\textbf{h})$ we have
$$\lim_{\|\textbf{h}\|\to\infty}\|f_\textbf{h}\|_{\ell^p(\Z^n)}=2^{1/p}\|f\|_{\ell^p(\Z^n)},\;\;\;\lim_{\|\textbf{h}\|\to\infty}\|A^\mathbb{P}_Nf_\textbf{h}\|_{\ell^q(\Z^n)}=2^{1/q}\|A^\mathbb{P}_N f\|_{\ell^q(\Z^n)}.$$
This would in turn force the existence for each $\epsilon>0$ of some $f\in\ell^p(\Z^n)$ (not the zero function) and $\textbf{h}$ such that
$$\|A^\mathbb{P}_Nf_\textbf{h}\|_{\ell^q(\Z^n)}\ge \|f_\textbf{h}\|_{\ell^p(\Z^n)}(\| A^\mathbb{P}_N \|_{\ell^p(\Z^n)\mapsto\ell^{q}(\Z^n)}-\epsilon)2^{\frac1q-\frac1p},$$
leading to a contradiction.

\bigskip

Our proof of Theorem \ref{thm:main}  relies on obtaining suitable estimates for the corresponding Fourier multiplier. For this purpose we use the Hardy-Littlewood circle method and estimates for the exponential sums from \cite{bourg}.
\bigskip

In Section \ref{sec:maineps} we first prove a version of Theorem \ref{thm:main} which covers also the endpoint $p=\frac{n+3}{n+1}$, however with the $\eps$-loss in the power of $N$.
\begin{thm}\label{thm:maineps}
Let $\frac{n+3}{n+1}\le p\le2$. For any $\epsilon>0$ the following bound holds
\begin{equation}\label{eq:maineps}
\| A^\mathbb{P}_N f\|_{\ell^{p'}(\Z^n)}\lesssim_\epsilon N^\epsilon N^{-(n+1)(\frac{2}p-1)}\|f\|_{\ell^p(\Z^n)}.
\end{equation}
\end{thm}

In Section \ref{sec:main} we adopt Bourgain's argument from \cite{bourg} to remove the $N^\epsilon$ factor from the estimate in Theorem \ref{thm:maineps} at the expense of moving away from the endpoint. The interesting question about the validity of \eqref{eq:main} for $p=\frac{n+3}{n+1}$ remains open.
\bigskip

There are similarities between the $l^p$ improving problem considered here and the discrete restriction estimate for the paraboloid, first considered in the landmark paper \cite{bourg}. This restriction problem is about proving sharp estimates of the form
\begin{equation}
\label{TT}
\|g\ast \widehat{K_N^\mathbb{P}}\|_{L^{p'}(\T^n)}\lesssim N^{\alpha_p}\|g\|_{L^{p}(\T^n)}
\end{equation}
for $1\le p\le 2$.
Partial progress on this problem has been made in \cite{bourg} by combining the circle method with $L^1\to L^\infty$ and $L^2\to L^2$ interpolation, similar to what we do in this paper. However, this method could not yield the full range of estimates \eqref{TT} in any dimension. Instead, the restriction problem has been solved in \cite{BD} (in all dimensions, apart from $N^\epsilon$  losses) using $\ell^2$ decoupling. This reduction to decoupling was possible in part because of the $TT^*$ method. Indeed, \eqref{TT} is easily seen to be equivalent with the $L^2$ based inequality
$$\|\widehat{K_N^\mathbb{P} f}\|_{L^{p'}(\T^n)}\lesssim N^{\frac{\alpha_p}{2}}\|f\|_{L^{2}(\Z^n)}.$$
A similar reduction is not possible for $\ell^p$ improving, as  the operator $f\mapsto f\ast {K_N^\mathbb{P}}$ is not positive, thus not of the form $TT^*$.

\subsection*{Notation}

Throughout the paper we use standard notation with all symbols referring to the spaces $\Z^n$ and $\T^n:=[0,1)^n$. Further, we write $\ast$ for the convolution on  $\Z^n$.  We set $\N = \{1,2, \ldots\}$, $\mathcal{D}=\{2^m:m\in\Z\}$. Moreover, we let $e(t)=e^{2\pi i t}$ and use the following notation for the Fourier transform on $\Z^n$
\begin{equation*}
\widehat{f} (\xi)= \sum_{m \in \Z^n} f(m) e(m \xi), \qquad f \in \ell^2(\Z^n), \quad \xi \in \T^n.
\end{equation*}

While writing estimates, we will use the notation $X \lesssim Y$
to indicate that $X \le CY$ with a positive constant $C$ independent of significant quantities.
We shall write $X \simeq Y$ when simultaneously $X \lesssim Y$ and $Y \lesssim X$.

\section{Hardy-Littlewood decomposition and the proof of Theorem \ref{thm:maineps}}\label{sec:maineps}
For $N\in\N$ let $\sigma_N:\Z\rightarrow [0,1]$ be a function satisfying $\mathbbm{1}_{(-N,N)}\le\sigma_N\le \mathbbm{1}_{(-2N,2N)}$ and such that $s_k=\sigma_N(k+1)-\sigma_N(k)$ is bounded by $1/N$ and has total variation bounded by $1/N$
$$\sum_{k\in\Z} | s_{k+1}-s_{k}|\le \frac1N.$$
Define
\begin{align*}
A_N f(x)=\frac{1}{N^{n-1}}K_N\ast f(x),
\end{align*}
where
\begin{align*}
K_N(x)=\sum_{k_1\in\Z}\dots \sum_{k_{n-1}\in\Z}\sigma_N(k_1)\dots \sigma_N(k_{n-1}) \delta_{(k_1,\dots,k_{n-1}, k_1^2+\dots+k_{n-1}^2)}(x).
\end{align*}
Note that if $f\ge 0$ then $A^{\mathbb{P}}_N f(x)\le A_N f(x)$ for every $x\in\Z^n$, so to prove Theorem \ref{thm:maineps} (and also the sufficiency part of Theorem \ref{thm:main}) one can replace $A^\mathbb{P}_N$ with $A_N$.  The technical assumptions imposed on $\sigma_N$ are necessary in order to get a suitable Gauss sum estimate, see \eqref{eq:gauss} below.

The Fourier transform of the kernel is given for $\xi=(\xi_1,\dots,\xi_n)\in\mathbb{T}^n$ by
\begin{align*}
m_N(\xi)&=\widehat{K_N}(\xi)=\sum_{k_1\in\Z}\dots \sum_{k_{n-1}\in\Z} \sigma_N(k_1)\dots \sigma_N(k_{n-1})\\
&\times e(\xi_1 k_1+...+\xi_{n-1}k_{n-1}+\xi_n (k_1^2+\dots+k_{n-1}^2))\\
&=\prod_{i=1}^{n-1}\left(\sum_{k\in\Z} \sigma_N(k) e(\xi_i k+\xi_n k^2)\right)=\prod_{i=1}^{n-1} G(\xi_n, \xi_i),
\end{align*}
where
\begin{align*}
G(t,y)=\sum_{k\in\Z} \sigma_N(k) e(y k+t k^2).
\end{align*}
Recall that the following estimate holds
\begin{equation}\label{eq:gauss}
|G(t,y)|\lesssim \frac{1}{\sqrt{q}}\min\{N, \frac{1}{\sqrt{|t-a/q|}}\}
\end{equation}
uniformly in $y\in\T$ and $|t-a/q|<10/(qN)$, see \cite[Lemma 3.18]{bourg}.

We shall partition $\mathbb{T}$ into the so called major and minor arcs. For $1\le q\le N/10$ and $a\in A_q$ with
$A_q=\{1\le a\le q-1:\;(a,q)=1\}$,
consider
$$
I(q,N,a):=\left[\frac{a}{q}-\frac{1}{qN}, \frac{a}{q}+\frac{1}{qN}\right]
$$
and
$$
4I(q,N,a):=\left[\frac{a}{q}-\frac{4}{qN}, \frac{a}{q}+\frac{4}{qN}\right].
$$
Observe that the sets $4I(q,N,a)$ are mutually disjoint for  $q\le N/10$.

Let $\psi\in C_c^\infty(\mathbb{R}^n)$ be such that $\mathbbm{1}_{(-1,1)}\le \psi \le\mathbbm{1}_{(-2,2)}$. Then let $\varphi(t):=\psi(t)-\psi(2t)$ and $\varphi_0:=\psi$. Observe that then for each $\xi_n\in I(q,N,a)$
$$
\varphi_0\left(N^2(\xi_n-a/q)\right)+\sum_{1\le 2^l\le N/q} \varphi\left(2^lNq(\xi_n-a/q)\right)=1.
$$
Now let
\begin{align*}
\eta_{l,a,q}(\xi_n)&=\varphi\left(2^l Nq(\xi_n-a/q)\right)-\varphi\left(2^l Nq(\xi_n-a/q-3/(Nq))\right),\\
\eta^0_{a,q}(\xi_n)&=\varphi_0\left(N^2(\xi_n-a/q)\right)-\varphi_0\left(N^2(\xi_n-a/q-3/(Nq))\right).\\
\end{align*}
This construction is meant to guarantee the mean zero property
\begin{equation}
\label{eq:meanzero}
\int_\R \eta_{l,a,q}(t)dt=\int_\R\eta^0_{a,q}(t)dt=0.
\end{equation}
Note that
\begin{align*}
\supp \eta_{l,a,q}&\subset \{\xi_n\in\mathbb{T}:|\xi_n-a/q|\simeq \frac{1}{2^l Nq}\}\cup \{\xi_n\in\mathbb{T}:|\xi_n-a/q|\simeq \frac{1}{Nq}\},\\
\supp \eta^0_{a,q}&\subset \{\xi_n\in\mathbb{T}:|\xi_n-a/q|\lesssim\frac{1}{N^2}\}\cup \{\xi_n\in\mathbb{T}:|\xi_n-a/q|\simeq \frac{1}{Nq}\}.
\end{align*}
Moreover, as $q$ ranges from $1$ to $N/10$, $a\in A_q$ and $1\le 2^l <N/Q$ all the supports above are mutually disjoint.
We will see that the addition of the extra bumps to the functions $\eta_{l,a,q}$ and $\eta^0_{a,q}$ does not harm the contribution from the minor arcs.

For further reference we note that the Fourier transform of $\eta_{l,a,q}$, as a function on $\R$, is given by
\begin{equation}\label{eq:FTeta}
\widehat{\eta_{l,a,q}}(t)=\frac{1}{2^l Nq}\widehat{\varphi}\left(\frac{t}{2^lNq}\right)\left[e\left(\frac{a}{q}t\right)-e\left(\left(\frac{a}{q}+\frac{3}{qN}\right)t\right)\right],\qquad t\in\R.
\end{equation}

For a dyadic $1\le Q\le N/10$ and $1\le 2^l\le N/Q$ define
\begin{align*}
m_{Q,l}(\xi)&=m_N(\xi) \sum_{Q/2\le q\le Q}\sum_{a\in A_q}\eta_{l,a,q}(\xi_n),\\
m^0_{Q}(\xi)&=m_N(\xi) \sum_{Q/2\le q\le Q}\sum_{a\in A_q}\eta^0_{a,q}(\xi_n).
\end{align*}

Decompose
\begin{align*}
m_N(\xi)=m^{\textrm{min}}_N(\xi)+m^{\textrm{maj}}_N(\xi),
\end{align*}
where
\begin{align*}
m^{\textrm{maj}}_N(\xi)=\sum_{\substack{Q\in\mathcal{D}\\ 1\le Q\le N/10}}\left(m^0_{Q}(\xi)+\sum_{1\le 2^l\le N/Q} m_{Q,l}(\xi)\right)
\end{align*}
and
\begin{align*}
m^{\textrm{min}}_N(\xi)=m_N(\xi)-m^{\textrm{maj}}_N(\xi).
\end{align*}
Note that
\begin{equation}\label{eq:suppmin}
\supp{m^{\textrm{min}}_N}\subset \T\setminus\left(\bigcup_{q\le N/10}\bigcup_{a\in A_q}I(q,N,a)\right).
\end{equation}
\subsection{Major arcs estimates}
For $k\in \Z$ let $d(k)$ denote the number of divisors of $k$.  For $k\in\Z$, $Q\in\N$ let $d(k, Q)$ denote the number of positive divisors of $k$ which are smaller than $Q$.

We will need the following auxiliary estimate, whose proof can be found in \cite{bourg}.
\begin{lem}\label{lem:bourg}\cite[Lemma 3.33]{bourg}
For any $\epsilon>0$ we have
\begin{equation}
\label{uyuyuy1}
\sum_{Q/2\le q\le Q }\left|\sum_{a\in A_q}e\left(\frac{a}{q} k\right)\right|\lesssim_{\epsilon} Q^{1+\epsilon} d(k,Q).
\end{equation}
\end{lem}
Since
\begin{equation}
\label{hhchhhhdduieyduey}
d(k)\lesssim_{\epsilon} k^\epsilon,
\end{equation}
as an immediate consequence of the above result we get the following estimate.
\begin{cor}\label{cor:bourg}
Let $Q\in\N$, $k\in\Z\setminus\{0\}$. Then for any $\epsilon>0$
\begin{align*}
\sum_{Q/2\le q\le Q }\left|\sum_{a\in A_q}e\left(\frac{a}{q} k\right)\right|\lesssim_{\epsilon} Q (Qk)^{\epsilon}.
\end{align*}
\end{cor}
The main result of this subsection reads as follows.
\begin{lem}\label{lem:maj}
For every $\epsilon>0$ the following estimates hold:

\begin{align}\label{eq:est1}
\|m_{Q,l}\|_{L^\infty(\T^n)}&\lesssim (N 2^l)
^{(n-1)/2},\\ \label{eq:est2}
\|\widehat{m_{Q,l}}\|_{\ell^\infty(\Z^n)}&\lesssim_\epsilon (N 2^l)
^{-1} (QN)^\epsilon,\\ \label{eq:est3}
\|m^0_{Q}\|_{L^\infty(\T^n)}&\lesssim (N^2/Q)
^{(n-1)/2},\\ \label{eq:est4}
\|\widehat{m^0_{Q}}\|_{\ell^\infty(\Z^n)}&\lesssim_\epsilon (N^2 /Q)
^{-1} (QN)^\epsilon.
\end{align}
\end{lem}
\begin{proof}
Since
$$
\supp m_{Q,l}\subset \bigcup_{Q/2\le q\le Q}\bigcup_{a\in A_q} \{\xi_n\in\mathbb{T}:|\xi_n-a/q|\simeq \frac{1}{2^l Nq}\}\cup \{\xi_n\in\mathbb{T}:|\xi_n-a/q|\simeq \frac{1}{Nq}\},
$$
the estimate \eqref{eq:gauss} implies
$$
|G(\xi_n, \xi_i)|\lesssim  (2^l N)^{1/2},
$$
which gives part \eqref{eq:est1} of the claim, since $m_N(\xi)=\prod_{i=1}^{n-1}G(\xi_n, \xi_i)$.

To prove \eqref{eq:est2} we begin with writing, for $r=(r',r_n)\in\Z^n$ with $r'=(r_1,r_2,\ldots,r_{n-1})\in\Z^{n-1}$,
\begin{align*}
\widehat{m_{Q,l}}(r)&=\int_{\T^n}m_{Q,l}(\xi)e(-r\xi)d\xi\\&=\int_\T \sum_{Q/2\le q\le Q}\sum_{a\in A_q}\eta_{l,a,q}(\xi_n)
\left(\prod_{i=1}^{n-1}\int_\T G(\xi_n, \xi_i)e(-r_i \xi_i) d\xi_i\right) e(-r_n \xi_n)d\xi_n\\&=\int_\T \sum_{Q/2\le q\le Q}\sum_{a\in A_q}\eta_{l,a,q}(\xi_n) \left(\prod_{i=1}^{n-1}\sum_{k\in\Z} \sigma_N(k)e(k^2 \xi_n) \int_\T e((k-r_i) \xi_i) d\xi_i\right) e(-r_n \xi_n)d\xi_n\\&=\int_\T \sum_{Q/2\le q\le Q}\sum_{a\in A_q}\eta_{l,a,q}(\xi_n) \left(\prod_{i=1}^{n-1} \sigma_N(r_i)e(r_i^2 \xi_n) \right) e(-r_n \xi_n)d\xi_n
\\
&=\sigma_N(r_1)\dots  \sigma_N(r_{n-1})\sum_{Q/2\le q\le Q}\sum_{a\in A_q}\int_\T \eta_{l,a,q}(\xi_n)  e((|r'|^2-r_n) \xi_n)d\xi_n\\
&= \sigma_N(r_1)\dots  \sigma_N(r_{n-1})\sum_{Q/2\le q\le Q}\sum_{a\in A_q}\widehat{\eta_{l,a,q}}(|r'|^2-r_n).
\end{align*}
We distinguish two cases. If $|r'|^2=r_n$, then the above computation combined with \eqref{eq:meanzero} shows that $\widehat{m_{Q,l}}(r)=0$.

On the other hand, if $||r'|^2-r_n|\ge 1$, using the representation \eqref{eq:FTeta}, the Schwartz decay of the function $\widehat{\eta}$ and Corollary \ref{cor:bourg} we obtain
\begin{align*}
|\widehat{m_{Q,l}}(r)|&\le\sigma_N(r_1)\dots  \sigma_N(r_{n-1})\sum_{Q/2\le q\le Q}\sum_{a\in A_q}\left|\widehat{\eta_{l,a,q}}(|r'|^2-r_n)\right|\\
&\lesssim
\sigma_N(r_1)\dots  \sigma_N(r_{n-1})\frac{1}{2^l NQ}\left(1+\frac{||r'|^2-r_n|}{2^l NQ}\right)^{-100n}\sum_{Q/2\le q\le Q}\left|\sum_{a\in A_q}e\left(\frac{a}{q}\left(|r'|^2-r_n\right)\right)\right|\\
&\lesssim_\epsilon
\frac{1}{2^l N} (QN)^{\epsilon}.
\end{align*}
In the last estimate above we use the decay of $\left(1+\frac{||r'|^2-r_n|}{2^l NQ}\right)^{-100n}$ if $\left||r'|^2-r_n\right|>N^3$ and Corollary \ref{cor:bourg} if $\left||r'|^2-r_n\right|\le N^3$.

The arguments proving \eqref{eq:est3} and \eqref{eq:est4} are analogous.

\end{proof}

\subsection{Minor arcs estimates}
\begin{lem}\label{lem:min}
For every $\epsilon>0$ the following estimates hold
\begin{align}\label{eq:min1}
\|m^{\textrm{min}}_N\|_{L^\infty(\T^n)}&\lesssim_\epsilon N^{(n-1)/2}N^\epsilon,\\ \label{eq:min2}
\|\widehat{m^{\textrm{min}}_N}\|_{\ell^\infty(\Z^n)}&\lesssim_\epsilon N^\epsilon .
\end{align}
\end{lem}

\begin{proof}
By the Dirichlet's Principle, for each $\xi_n\in [0,1)$ there exists  $1\le q\le N-1$ and $a\in A_q$ such that
$$
|\xi_n-a/q|<\frac{1}{qN}.
$$
If $\xi_n\in\supp m_N^{\textrm{min}}$, then condition \eqref{eq:suppmin} implies that  $q>N/10$.  Therefore, $q\simeq N$ and so using \eqref{eq:gauss} we get
$$
|G(\xi_n, \xi_i)|\lesssim N^\epsilon N^{1/2},
$$
thus \eqref{eq:min1} is proved.

To get \eqref{eq:min2}, we use \eqref{eq:est2} and \eqref{eq:est4} and write for any $\epsilon>0$
\begin{align*}
\|\widehat{m^{\textrm{maj}}_{N}}\|_{\ell^\infty(\Z^n)}&\le \sum_{\substack{Q\in\mathcal{D}\\ Q\le N/10}}\left(\|\widehat{m^0_{Q}}\|_{\ell^\infty(\Z^n)}+\sum_{1\le 2^l\le N/Q} \|\widehat{m_{Q,l}}\|_{\ell^\infty(\Z^n)}\right)\\
&\lesssim_{\epsilon}
N^{\epsilon}\sum_{\substack{Q\in\mathcal{D}\\ Q\le N/10}}\left((N^2 /Q)
^{-1} +\sum_{1\le 2^l\le N/Q}(N 2^l)^{-1}\right)\\&\lesssim_{\epsilon} N^\epsilon.
\end{align*}
Combining this with the trivial observation that $\|\widehat{m_N}\|_{\ell^\infty(\Z^n)} =1$, we obtain
\begin{equation*}
\|\widehat{m^{\textrm{min}}_{N}}\|_{\ell^\infty(\Z^n)}\le
\|\widehat{m^{\textrm{maj}}_{N}}\|_{\ell^\infty(\Z^n)}+
\|\widehat{m_{N}}\|_{\ell^\infty(\Z^n)}\lesssim_\epsilon N^\epsilon.
\end{equation*}

\end{proof}

\subsection{$\ell^p\rightarrow \ell^{p'}$ estimates}
We begin with deriving $\ell^p(\Z^n)\rightarrow \ell^{p'}(\Z^n)$ inequalities which are consequences of the estimates from the previous subsections and linear interpolation. More precisely, we will use the fact that for each kernel $K$ we have
\begin{equation}
\label{eq:l^2}
\|K\ast f\|_2\lesssim \|\widehat{K}\|_{\infty}\|f\|_2
\end{equation}
and
\begin{equation}
\label{eq:l^infty}
\|K\ast f\|_\infty\lesssim \|K\|_{\infty}\|f\|_1.
\end{equation}
\begin{cor}\label{cor:maj}
Let $1<p\le 2$. For every $\epsilon>0$ the following estimates hold:
\begin{align} \label{eq:maj1}
\|\widehat{m_{Q,l}}\ast f\|_{\ell^{p'}(\Z^n)}&\lesssim_\epsilon N^\epsilon(N 2^l)^{\frac{n+1}{p'}-1}\|f\|_{\ell^p(\Z^n)},\\ \label{eq:maj2}
\|\widehat{m^0_{Q}}\ast f\|_{\ell^{p'}(\Z^n)}&\lesssim_\epsilon N^\epsilon(N^2/Q)^{\frac{n+1}{p'}-1}\|f\|_{\ell^p(\Z^n)}.
\end{align}
Moreover, if $\frac{n+1}{p'}-1\ge 0$, then
\begin{equation}
\label{eq:maj3}
\|\widehat{m^{\textrm{maj}}_{N}}\ast f\|_{\ell^{p'}(\Z^n)}\lesssim_\epsilon N^\epsilon N^{2\frac{n+1}{p'}-2}\|f\|_{\ell^p(\Z^n)}.
\end{equation}

\end{cor}
\begin{proof}
Parts \eqref{eq:maj1} and \eqref{eq:maj2} follow immediately by interpolating $\ell^2\mapsto \ell^2$ and $\ell^1\mapsto \ell^\infty$ bounds for the convolution operator, using \eqref{eq:l^2}, \eqref{eq:l^infty} and Lemma \ref{lem:maj}.

Summing up the estimates \eqref{eq:maj1} and \eqref{eq:maj2} we get
\begin{align*}
\|\widehat{m^{\textrm{maj}}_{N}}\ast f\|_{\ell^{p'}(\Z^n)}&\le \sum_{\substack{Q\in\mathcal{D}\\ Q\le N/10}}\left(\|\widehat{m^0_{Q}}\ast f\|_{\ell^{p'}(\Z^n)}+\sum_{1\le 2^l\le N/Q} \|\widehat{m_{Q,l}}\ast f\|_{\ell^{p'}(\Z^n)}\right)\\
&\lesssim_\epsilon
N^{\epsilon}\sum_{\substack{Q\in\mathcal{D}\\ Q\le N/10}}\left((N^2/Q)^{\frac{n+1}{p'}-1} +\sum_{1\le 2^l\le N/Q}(N 2^l)^{\frac{n+1}{p'}-1}\right)\|f\|_{\ell^p(\Z^n)}\\
&\lesssim_\epsilon N^\epsilon N^{2\frac{n+1}{p'}-2}\|f\|_{\ell^p(\Z^n)},
\end{align*}
provided that $\frac{n+1}{p'}-1\ge 0.$
\end{proof}

\begin{cor}\label{cor:min}
Let $1\le p\le 2$. For any $\epsilon>0$ the following estimate holds
\begin{equation}
\label{eq:min}
\|\widehat{m^{\textrm{min}}_{N}}\ast f\|_{\ell^{p'}(\Z^n)}\lesssim_\epsilon N^\epsilon N^{\frac{n-1}{p'}}\|f\|_{\ell^p(\Z^n)}.
\end{equation}
\end{cor}
\begin{proof}
Interpolate the bounds from Lemma \ref{lem:min}.
\end{proof}
Now we are ready to prove the main result of this section.
\begin{proof}[Proof of Theorem \ref{thm:maineps}]
Observe that
\begin{align*}
A_N f=\frac{1}{N^{n-1}}\left(R(\widehat{m^{\textrm{maj}}_{N}})\ast f+R(\widehat{m^{\textrm{min}}_{N}})\ast f\right),
\end{align*}
where $R$ is the reflection operator $Rg(x)=g(-x)$.

Therefore, since $\frac{n-1}{p'}\le 2\frac{(n+1)}{p'}-2$ if and only if $p\ge\frac{n+3}{n+1}$, we can apply Corollaries \ref{cor:maj} and \ref{cor:min} to get
\begin{align*}
\|A_N f\|_{\ell^{p'}(\Z^n)}&\lesssim_\epsilon N^\epsilon N^{2\frac{(n+1)}{p'}-2-(n-1)}\|f\|_{\ell^p(\Z^n)}\\&= N^\epsilon N^{-(n+1)(\frac{2}p-1)}\|f\|_{\ell^p(\Z^n)},
\end{align*}
for $p\in [\frac{n+3}{n+1},2]$.
\medskip

\end{proof}

We note that since $\frac{n-1}{p'}-(n-1)<-(n+1)(\frac{2}{p}-1)$ when $p>\frac{n+3}{n+1}$, the minor arc contribution \eqref{eq:min} is better than the global contribution \eqref{eq:main}. Because of this, the presence of the $N^\epsilon$ term in  \eqref{eq:min} is not a serious issue and will cause no trouble in the remaining part of the paper. However, the $N^\epsilon$ term in the estimates for the major arcs needs to be addressed carefully. The main sources of the $N^\epsilon$ term are Lemma \ref{lem:bourg} and \eqref{hhchhhhdduieyduey}. In the next section we will still use this lemma, but we will refine \eqref{hhchhhhdduieyduey}.

\section{$\epsilon$-removal technology and the proof of Theorem \ref{thm:main}}\label{sec:main}
Note that, same as in Section \ref{sec:maineps}, when proving the sufficiency part of Theorem \ref{thm:main} one can consider $A_N$ instead of $A^{\mathbb{P}}_N$.
We begin with improving the major arc estimate \eqref{eq:maj3}. Recall the definitions from the previous section.
\begin{align*}
m_{Q,l}(\xi)&=m_N(\xi) \sum_{Q/2\le q\le Q}\sum_{a\in A_q}\eta_{l,a,q}(\xi_n),\\
m^0_{Q}(\xi)&=m_N(\xi) \sum_{Q/2\le q\le Q}\sum_{a\in A_q}\eta^0_{a,q}(\xi_n),\\
m^{\textrm{maj}}_N(\xi)&=\sum_{\substack{Q\in\mathcal{D}\\ 1\le Q\le N/10}}\left(m^0_{Q}(\xi)+\sum_{1\le 2^l\le N/Q} m_{Q,l}(\xi)\right),\\
m^{\textrm{min}}_N(\xi)&=m_N(\xi)-m^{\textrm{maj}}_N(\xi).
\end{align*}
To obtain an improvement of the estimates from the previous section we need some auxiliary results. The first of them is a version of \cite[Lemma 3.47]{bourg}. This may be seen as a refinement of \eqref{hhchhhhdduieyduey}.

\begin{lem}\label{lem:divisors}
Let $\tau, B>0$. Then the following estimate holds uniformly over $Q,N\in\N$ and $D>0$
\begin{equation}\label{eq:divisors}
|\{1\le k\le N: d(k, Q)>D\}|\lesssim_{\tau, B} D^{-B}Q^{\tau}N.
\end{equation}
\end{lem}
\begin{rem}
Note that compared to Bourgain's \cite[Lemma 3.47]{bourg} we do not include the term $d(0,Q)$ corresponding to $k=0$ on the left hand side of the estimate. As it shall soon become apparent, this term does not appear in our analysis due to the application of the mean zero property \eqref{eq:meanzero}. For reader's convenience, we provide the proof below.
\end{rem}
\begin{proof}
We may assume that $B$ is a positive integer.	
Write for $1\le q\le Q$
$$\mathcal{I}_q(k)=\begin{cases}1,\;\;\text{if }q|k,\\0,\;\;\text{otherwise.}\end{cases}$$
Then, denoting by $[q_1,\ldots,q_B]\in\{1,2,\ldots,Q^B\}$ the least common multiple of $q_1,\ldots,q_B$ we get
\begin{align*}
|\{1\le k\le N: d(k, Q)>D\}|&\le D^{-B}\sum_{k=1}^N \left(\sum_{q=1}^Q\mathcal{I}_q(k)\right)^B\\&= D^{-B}\sum_{q_1=1}^Q\ldots\sum_{q_B=1}^Q|\{1\le k\le N:[q_1,\ldots,q_B]\text{ divides }k\}|\\&\le D^{-B}\sum_{q_1=1}^Q\ldots\sum_{q_B=1}^Q\frac{N}{[q_1,\ldots,q_B]}\\&\le ND^{-B}\sum_{q=1}^{Q^B}\frac{d(q)^B}{q}\\&\le C_{\tau,B}ND^{-B}Q^{\tau},
\end{align*}
where the last bound follows by \eqref{hhchhhhdduieyduey}.
\end{proof}

We shall also need the following consequence of Lemma \ref{lem:divisors}.
\begin{lem}\label{lem:divisorssum}\cite[equation $(3.72)$]{bourg}
Let $\tau, B>0$ be any given constants. Then the following estimate holds uniformly over $K,Q,N\in\N$ and $D>0$
$$|\{(r_1,\dots, r_n): |r_1|,\dots, |r_{n-1}|
\le N, |r_n|\le K, r_n-r_1^2-\dots-r^2_{n-1}\neq 0:  d(r_n-r_1^2-\dots-r^2_{n-1}, Q)>D\}|
$$
\begin{equation}\label{eq:divisorssum}\lesssim_{\tau, B} D^{-B}Q^{\tau}\max(K,N^2 )N^{n-1}.
\end{equation}
\end{lem}
\begin{proof}
We need to observe two things. First, $|r_n-r_1^2-\dots-r^2_{n-1}|\lesssim \max(K,N^2 ).$ Second, the equation $r_n-r_1^2-\dots-r^2_{n-1}=k$ has $O(N^{n-1})$ solutions.

\end{proof}
The above number theoretic lemmas allow for a more delicate treatment of the expression arising from computing
$\widehat{m_{Q,l}}$.
\begin{prop}\label{prop:improvedest}
For any $B, \kappa>0$ the following bounds hold uniformly over $D,N,Q\ge 1$, $l\ge 0$ and  $f$
\begin{align}\label{eq:e1}
\|f\ast\widehat{m_{Q,l}}\|_{\ell^2(\Z^n)}&\lesssim (N 2^l)
^{(n-1)/2}\|f\|_{\ell^2(\Z^n)},
\\ \label{eq:e2'}
\|f\ast\widehat{m_{Q,l}}\|_{\ell^\infty(\Z^n)}&\lesssim_{B,\kappa} Q^{2+2\kappa}N^n (2^l Q)^{-1}D^{-B}\|f\|_{\ell^\infty(\Z^n)}+\frac{D Q^{\kappa}}{N 2^l }\|f\|_{\ell^1(\Z^n)},
\\
 \label{eq:e3}
\|f\ast \widehat{m^0_{Q}}\|_{\ell^2(\Z^n)}&\lesssim (N^2/Q)
^{(n-1)/2}\|f\|_{\ell^2(\Z^n)},\\
\label{eq:e4'}
\|f\ast\widehat{m^0_{Q}}\|_{\ell^\infty(\Z^n)}&\lesssim_{B,\kappa}  Q^{2+2\kappa}N^{n-1}D^{-B}\|f\|_{\ell^\infty(\Z^n)}+\frac{D Q^{1+\kappa}}{N^2}\|f\|_{\ell^1(\Z^n)}.
\end{align}
\end{prop}
\medskip

\begin{rem}
The novelty of \eqref{eq:e2'} and \eqref{eq:e4'} compared to their counterparts from Lemma \ref{lem:maj} is the lack of the $N^\epsilon$ term, which is substituted with the flexible variable $D$. This comes at the expense of introducing an extra term involving $\|f\|_{\ell^\infty(\Z^n)}$, that will prove to be harmless.
\end{rem}
\begin{proof}
Estimates \eqref{eq:e1} and \eqref{eq:e3}
follow from Lemma \ref{lem:maj}. It remains to prove \eqref{eq:e2'}, the argument for \eqref{eq:e4'} being analogous.
 Recall that in the proof of Lemma \ref{lem:maj} we showed that if $r=(r',r_n)\in\Z^n$ with $r'=(r_1,r_2,\ldots,r_{n-1})\in\Z^{n-1}$ is such that $|r'|^2=r_n$, then  $\widehat{m_{Q,l}}(r)=0$.

Therefore we can assume that $||r'|^2-r_n|\ge 1$, in which case we can estimate
$$
|\widehat{m_{Q,l}}(r)|\lesssim_\kappa
$$
$$
\sigma_N(r_1)\dots  \sigma_N(r_{n-1})\frac{1}{2^l NQ}\left(1+\frac{||r'|^2-r_n|}{2^l NQ}\right)^{-1/\kappa}\sum_{Q/2\le q\le Q}\left|\sum_{a\in A_q}e\left(\frac{a}{q}\left(|r'|^2-r_n\right)\right)\right|.
$$
Let us now fix a large constant $C>0$  and decompose
\begin{align*}
f\ast\widehat{m_{Q,l}}(x)=\sum_{\substack{|r_1|,\dots,|r_{n-1}|\le 2N \\ |r_n|\le C Q^\kappa N^2\\ ||r'|^2-r_n|\ge 1}}f(x-r)\widehat{m_{Q,l}}(r)+\sum_{\substack{|r_1|,\dots,|r_{n-1}|\le 2N \\ |r_n|\ge C  Q^\kappa N^2}}f(x-r)\widehat{m_{Q,l}}(r)=:I_1+I_2.
\end{align*}
To bound the second term we use  the trivial bound
$$\sum_{Q/2\le q\le Q}\left|\sum_{a\in A_q}e\left(\frac{a}{q}\left(|r'|^2-r_n\right)\right)\right|\le Q^2$$
and the inequality  $\left||r'|^2-r_n\right|\gtrsim Q^\kappa N^2\gtrsim 2^lNQ^{1+\kappa}$ to get
$$
|I_2(x)|\lesssim_{\kappa} \frac{1}{2^l N Q}\left(\frac{2^lNQ}{2^lNQ^{1+\kappa}}\right)^{1/\kappa}Q^2\|f\|_{\ell^1(\Z^n)}=\frac{1}{2^lN}\|f\|_{\ell^1(\Z^n)}.
$$
Note that clearly $\frac{1}{2^l N}\le \frac{DQ^{\kappa}}{2^l N}$, so the contribution from $I_2$ is controlled by the right-hand side of \eqref{eq:e2'}. Thus $I_2$ can be thought of as an error term.

It remains to deal with $I_1$.
Using Lemma \ref{lem:bourg} and then Lemma \ref{lem:divisorssum} (applied with $K=C N^2 Q^\kappa$) we get for any $x\in\Z^n$
\begin{align*}
|I_1(x)|&\lesssim\frac{1}{2^l NQ}\sum_{\substack{|r_1|,\dots,|r_{n-1}|\le 2N \\ |r_n|\le C Q^\kappa N^2\\ ||r'|^2-r_n|\ge 1}}|f(x-r)| \sum_{Q/2\le q\le Q}\left|\sum_{a\in A_q}e\left(\frac{a}{q}\left(|r'|^2-r_n\right)\right)\right|\\
&\lesssim _\kappa\frac{Q^{1+\kappa}}{2^l NQ}\sum_{\substack{|r_1|,\dots,|r_{n-1}|\le 2N \\ |r_n|\le C Q^\kappa N^2\\ ||r'|^2-r_n|\ge 1}}|f(x-r)|d(|r'|^2-r_n,Q)\\
&\lesssim_\kappa \frac{Q^{1+\kappa}}{2^l NQ}\left(D\sum_{\substack{|r_1|,\dots,|r_{n-1}|\le 2N \\ |r_n|\le C Q^\kappa N^2\\ ||r'|^2-r_n|\ge 1\\d(|r'|^2-r_n,Q)\le D}}|f(x-r)|+\sum_{\substack{|r_1|,\dots,|r_{n-1}|\le 2N \\ |r_n|\le C Q^\kappa N^2\\ ||r'|^2-r_n|\ge 1\\ d(|r'|^2-r_n,Q)> D}}|f(x-r)|d(|r'|^2-r_n,Q)\right)\\
&\lesssim_{\kappa,B}  \frac{Q^{\kappa}}{2^l N}\left(D\|f\|_{\ell^1(\Z^n)}+  D^{-B}Q^{1+\kappa} N^{n+1}\|f\|_{\ell^\infty(\Z^n)}\right),
\end{align*}
where in the last estimate we used a trivial bound $d(|r'|^2-r_n,Q)\le Q$.
Therefore \eqref{eq:e2'} is proved.
\medskip

 \end{proof}
Choosing suitably the values of the parameters we get the following corollary.
\begin{cor}\label{cor:improvedest}
For any $\tau>0$, $B>0$ and for any $\kappa>0$, the following estimates hold uniformly over $Q,N,M\ge 1$ and $l\ge 0$
\begin{align}
\label{eq:c2'}
\|f\ast\widehat{m_{Q,l}}\|_{\ell^\infty(\Z^n)}&\lesssim_{\kappa,\tau,B} N^n (2^l Q)^{-1}M^{-B}\|f\|_{\ell^\infty(\Z^n)}+\frac{M Q^{\kappa+\tau}}{N 2^l }\|f\|_{\ell^1(\Z^n)},
\\
\label{eq:c4'}
\|f\ast\widehat{m^0_{Q}}\|_{\ell^\infty(\Z^n)}&\lesssim_{\kappa,\tau,B} N^{n-1}M^{-B}\|f\|_{\ell^\infty(\Z^n)}+\frac{M Q^{1+\kappa+\tau}}{N^2}\|f\|_{\ell^1(\Z^n)}.
\end{align}
\end{cor}
\begin{proof}
Since $M\ge 1$, it suffices to  assume $B>(2+2\kappa)/\tau$. Take $D=M Q^\tau$ in \eqref{eq:e2'} and \eqref{eq:e4'}. It suffices to note that
\begin{equation*}
Q^{2+2\kappa}D^{-B}<M^{-B}.
\end{equation*}
\end{proof}
Finally, we are in a position to obtain the improvement of \eqref{eq:maj3}.
\begin{cor}\label{cor:improvedest}
Let $\frac{n+1}{n}< p\le 2$. Then for any $M\ge 1$ and  $B>0$ the following estimate holds
\begin{align}
&\|\widehat{m^{\textrm{maj}}_{N}}\ast f\|_{\ell^{p'}(\Z^n)}\lesssim_{B} N^{n-1} M^{-B}\|f\|_{\ell^{p'}(\Z^n)}+M N^{2(\frac{n+1}{p'}-1)}\|f\|_{\ell^{p}(\Z^n)} \label{eq:majfinal}
\end{align}
for any $f=\mathbbm{1}_E$, where $E\subset \Z^n$ is an arbitrary finite set.
\end{cor}
\begin{proof}

Fix any $\frac{n+1}{n}< p\le2$ and let $\theta=\frac{2}{p'}$. Then
\begin{align*}
\frac{1}p&=\frac{1-\theta}1+\frac{\theta}2,\\
 \frac{1}{p'}&=\frac{\theta}2+\frac{1-\theta}{\infty}
\end{align*}
and notice that for any characteristic function $f$  we have
\begin{align*}
\|f\|_{\ell^{2}(\Z^n)}^\theta \|f\|_{\ell^{\infty}(\Z^n)}^{1-\theta}&=\|f\|_{\ell^{p'}(\Z^n)},\\
\|f\|_{\ell^{2}(\Z^n)}^\theta \|f\|_{\ell^{1}(\Z^n)}^{1-\theta}&=\|f\|_{\ell^{p}(\Z^n)}.
\end{align*}
Due to H\"older's inequality, the equality sign in the above relations can be replaced with $\ge $ for arbitrary functions. However, in our case the inequality $\le$ will be needed, which justifies the use of characteristic functions.
\medskip

Interpolating \eqref{eq:e1} and \eqref{eq:e3} with \eqref{eq:c2'} and \eqref{eq:c4'}, respectively, where the latter two are applied with $\tau,\kappa>0$ such that $(1-\frac{2}{p'})(\tau+\kappa)<\frac{n+1}{p'}-1$ and $M_*=M^{\frac1{1-\theta}}$, we get
\begin{align}\nonumber
\|f\ast& \widehat{m_{Q,l}}\|_{\ell^{p'}(\Z^n)}\\&\lesssim\nonumber \|f\ast \widehat{m_{Q,l}}\|^\theta_{\ell^{2}(\Z^n)}\|f\ast \widehat{m_{Q,l}}\|^{1-\theta}_{\ell^{\infty}(\Z^n)}\\
&\lesssim_{B,\tau,\kappa} \nonumber
\left( (N 2^l)^{(n-1)/2}\|f\|_{\ell^2(\Z^n)}\right)^\theta\left(N^n (2^l Q)^{-1}M_*^{-B}\|f\|_{\ell^\infty(\Z^n)}+\frac{M_* Q^{1+\kappa+\tau}}{N 2^l Q}\|f\|_{\ell^1(\Z^n)}\right)^{1-\theta}\\
&\lesssim_B \label{eq:lpmajl}
N^{n-\frac{n+1}{p'}}M^{-B}Q^{-\sigma}(2^l Q)^{\frac{n+1}{p'}-1}\|f\|_{\ell^{p'}(\Z^n)}+Q^{-\sigma}(N 2^l Q)^{\frac{n+1}{p'}-1}M \|f\|_{\ell^{p}(\Z^n)},
\end{align}
where $\sigma:=\min\{\frac{n-1}{p'}, \frac{2}{p'}(\frac{n+1}2+\tau+\kappa)-1-\tau-\kappa\}>0.$
\medskip

Similarly, we get with the same $\sigma$ as above,
\begin{align}\label{eq:lpmaj0}
\|f\ast \widehat{m^0_{Q}}\|_{\ell^{p'}(\Z^n)}\lesssim_B
N^{n-1}M^{-B}Q^{-\sigma}\|f\|_{\ell^{p'}(\Z^n)}+Q^{-\sigma}(N^2)^{\frac{n+1}{p'}-1}M \|f\|_{\ell^{p}(\Z^n)}.
\end{align}

Summing up the estimates \eqref{eq:lpmajl} and \eqref{eq:lpmaj0} we get
\begin{align}
\|\widehat{m^{\textrm{maj}}_{N}}&\ast f\|_{\ell^{p'}(\Z^n)}\\&\le \sum_{\substack{Q\in\mathcal{D}\\ Q\le N/10}}\left(\|\widehat{m^0_{Q}}\ast f\|_{\ell^{p'}(\Z^n)}+\sum_{1\le 2^l\le N/Q} \|\widehat{m_{Q,l}}\ast f\|_{\ell^{p'}(\Z^n)}\right) \nonumber
\\\nonumber
&\lesssim
\sum_{\substack{Q\in\mathcal{D}\\ Q\le N/10}}\left(N^{n-1}M^{-B}Q^{-\sigma} +\sum_{1\le 2^l\le N/Q}N^{n-\frac{n+1}{p'}}M^{-B}Q^{-\sigma}(2^l Q)^{\frac{n+1}{p'}-1}\right)\|f\|_{\ell^{p'}(\Z^n)}
\\\nonumber
&+\sum_{\substack{Q\in\mathcal{D}\\ Q\le N/10}}\left(Q^{-\sigma}(N^2)^{\frac{n+1}{p'}-1}M +\sum_{1\le 2^l\le N/Q}Q^{-\sigma}(N 2^l Q)^{\frac{n+1}{p'}-1}M\right)\|f\|_{\ell^{p}(\Z^n)}
\\ \nonumber
&\lesssim_B N^{n-1} M^{-B}\|f\|_{\ell^{p'}(\Z^n)}+M N^{2(\frac{n+1}{p'}-1)}\|f\|_{\ell^{p}(\Z^n)},
\end{align}
provided that $\frac{n+1}{p'}-1\ge 0.$ It remains to notice that this condition is equivalent to $p\ge \frac{n+1}{n}.$ Note the the condition $\sigma>0$ insures that no additional logarithmic terms are introduced. 

\end{proof}
Now we are ready to present the proof of the main result of the paper. The argument relies on the ideas from \cite{bourg}.
\bigskip

\begin{proof}[Proof of Theorem \ref{thm:main}]
Note that it suffices to prove that for any $\frac{n+3}{n+1}<p\le 2$ and any $p< q<p'$ one has
\begin{equation}\label{eq:main*}
\| A_N f\|_{\ell^{q}(\Z^n)}\lesssim  N^{-(n+1)(1/p-1/q)}\|f\|_{\ell^p(\Z^n)}.
\end{equation}
Indeed, interpolating \eqref{eq:main*} with trivial estimate, see \eqref{trivial},
\begin{equation*}
\| A_N f\|_{\ell^{q}(\Z^n)}\le  \|f\|_{\ell^q(\Z^n)},\qquad q\in[1,\infty],
\end{equation*}
gives in particular \eqref{eq:main} for $\frac{n+3}{n+1}<p\le 2$, see Figure \ref{fig_1} below.
 \begin{figure}[h!]
\begin{tikzpicture}[scale=2.8]
\draw[arrows=-angle 60] (0,0) -- (0,2.5);
\draw[arrows=-angle 60] (0,0) -- (2.5,0);
\node at (-0.15,2.3) {$1/q$};
\node at (2.3,-0.1) {$1/p$};
\draw[very thin] (2.0,-0.05) -- (2.0,0.05);
\node at (2.0,-0.2) {$1$};
\draw[very thin] (1.6,-0.05) -- (1.6,0.05);
\node at (1.6,-0.2) {$\frac{n+1}{n+3}$};
\draw[very thin] (-0.05,2) -- (0.05,2);
\node at (-0.2,2) {$1$};
\draw[very thin] (-0.05,0.4) -- (0.05,0.4);
\node at (-0.2,0.4) {$\frac{2}{n+3}$};
\fill[black!5!white] (0,0)--(1.6,0.4) -- (2,2); 
\draw[thick, dashed] (0,0) -- (1.6,0.4);
\draw[thick, dashed] (2,2) -- (1.6,0.4);
\draw[very thin, dashed] (1.5,0.6) -- (0.6,0.6);
\draw[thick] (0,0) -- (2,2);
\draw[thick] (1.6,0.4) -- (1,1);
\node[rotate=45] at (0.95,1.05) {$q=p$};
\node[rotate=-45] at (1.3,0.8) {$q=p'$};
\filldraw[fill=white] (1.6,0.4) circle(1pt);
\filldraw[fill=black] (1.55,0.6) circle(1pt);
 \end{tikzpicture}
\caption{Visualization of the interpolation scheme used in the proof of Theorem \ref{thm:main}.} \label{fig_1}
\end{figure}
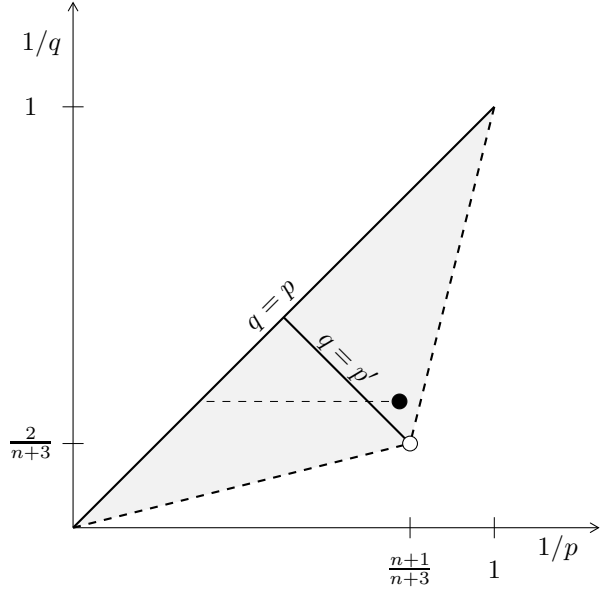
\bigskip

Fix $\frac{n+3}{n+1}<p\le 2$, $q<q_1<p'$ and let $f\in\ell^p(\Z^n)$ be positive and such that $\|f\|_{\ell^p(\Z^n)}=1$, which we clearly can assume without loss of generality. Moreover, for $\lambda>0$  define the level set
$$
E_\lambda=\{m\in\Z^n:A_N f(m)>\lambda\}
$$
and let $F=\mathbbm{1}_{E_\lambda}$.
Using positivity and then H\"older's inequality we obtain
\begin{align*}
N^{n-1}\lambda |E_\lambda|\le N^{n-1}\langle A_N f, F \rangle=\langle \widehat{m_N}\ast f, F \rangle=\langle f, R(\widehat{m_N})\ast F \rangle\le \|R(\widehat{m_N})\ast F\|_{\ell^{p'}(\Z^n)},
\end{align*}
where $R$ is the reflection operator $Rg(x)=g(-x)$.
\medskip

Combining \eqref{eq:min} with \eqref{eq:majfinal} we get  for each $M\ge 1$
\begin{align*}
\lambda |E_\lambda|&\le N^{-(n-1)}\|R(\widehat{m_N})\ast F\|_{\ell^{p'}(\Z^n)}\\&\lesssim_\epsilon M^{-B}|E_\lambda|^{1/p'}+\left(M N^{-(n+1)\left(\frac{1}{p}-\frac{1}{p'}\right)}+N^{-\frac{n-1}p+\eps}\right)|E_\lambda|^{1/p}\\
&\lesssim M^{-B}|E_\lambda|^{1/p'}+M N^{-(n+1)\left(\frac{1}{p}-\frac{1}{p'}\right)}|E_\lambda|^{1/p},
\end{align*}
where the last estimate holds provided that $\eps\le (n+3)\left(\frac{n+1}{n+3}-\frac{1}p\right)$.
\medskip

It follows that
\begin{equation*}
 |E_\lambda|\lesssim M^{-Bp}\lambda^{-p}+M^{p'} N^{-(n+1)\left(\frac{p'}{p}-1\right)}\lambda^{-p'}.
\end{equation*}
Let $\tau:=p'-q_1>0$ and take $M=N^{(n+1)\frac{\tau}{p p'}}\lambda^{\frac{\tau}{p'}}$. Note that  $M\ge 1$ if and only if $\lambda\ge N^{-\frac{n+1}p}$. Thus letting $B=p'\left(\frac{p'-\tau}{p}-1\right)\tau^{-1}=\frac{p'(q_1-p)}{p\tau}>0$ we have
$$M^{-Bp}\lambda^{-p}=M^{p'} N^{-(n+1)\left(\frac{p'}{p}-1\right)}\lambda^{-p'}=\lambda^{-q_1}N^{-(n+1)\left(\frac{q_1}{p}-1\right)}.$$
If  $\lambda\ge N^{-\frac{n+1}p}$ we get
\begin{equation}\label{eq:level}
 |E_\lambda|\lesssim \lambda^{-q_1}N^{-(n+1)\left(\frac{q_1}{p}-1\right)}.
\end{equation}
Note that by Tschebyshev's inequality and \eqref{trivial} we also have for any $\lambda\ge 0$
\begin{equation}\label{eq:level2}
|E_\lambda|\le \frac{\|A_N f\|^p_{\ell^p(\Z^n)}}{\lambda^p}\le \lambda^{-p}.
\end{equation}
Finally, applying the layer cake formula and then \eqref{eq:level2} and \eqref{eq:level} we obtain
\begin{align*}
q^{-1}\|A_N f\|^q_{\ell^q(\Z^n)}&= \int_0^{N^{-\frac{n+1}p}}\lambda^{q-1}|E_\lambda|d\lambda+\int_{N^{-\frac{n+1}p}}^\infty\lambda^{q-1}|E_\lambda|d\lambda\\
&\lesssim  \int_0^{N^{-\frac{n+1}p}}\lambda^{q-1-p}d\lambda+N^{-(n+1)\left(\frac{q_1}p-1\right)}\int_{N^{-\frac{n+1}p}}^\infty\lambda^{q-1-q_1}d\lambda\\
&\simeq N^{-(n+1)\left(\frac{q}p-1\right)},
\end{align*}
which gives the desired estimate.
\bigskip

It remains to prove the necessity part of the theorem. Letting
$$
f=\mathbbm{1}_{\{1,2,\dots,2N\}\times\dots\times \{1,2,\dots,2N\}\times \{1,2,\dots,n N^2\}},
$$
we get
\begin{align*}
\|f\|_{\ell^p(\Z^n)}&\simeq N^{\frac{n+1}p}
\end{align*}
and
\begin{align*}
\|A^{\mathbb{P}}_N f\|_{\ell^{p'}(\Z^n)}\gtrsim \left(\sum_{k_1=1}^N\dots \sum_{k_{n-1}=1}^N \sum_{k_n=1}^{N^2} 1\right)^{1/p'}=N^{\frac{n+1}{p'}},
\end{align*}
since
$$
A^{\mathbb{P}}_N f(x_1,\dots, x_n)=1,\quad\textrm{for}\quad x_i\in\{1,\dots, N\}, \quad i=1,\dots, n-1 \quad\textrm{and}\quad x_n\in\{1,\dots, N^2\}.
$$
This shows that
$$
\frac{\|A^{\mathbb{P}}_N f\|_{\ell^{p'}(\Z^n)}}{\|f\|_{\ell^p(\Z^n)}}\gtrsim N^{-(n+1)(\frac{1}{p}-\frac{1}{p'})}=N^{-(n+1)(\frac{2}{p}-1)}.
$$
Next, letting $f=\delta_0$ we get
\begin{align*}
\frac{\|A^{\mathbb{P}}_N f\|_{\ell^{p'}(\Z^n)}}{\|f\|_{\ell^p(\Z^n)}}&\ge \frac{\left(\sum_{k_1=1}^N\dots \sum_{k_{n-1}=1}^N  A^{\mathbb{P}}_N f(-k_1,\dots, -k_{n-1}, -(k_1^2+\dots+k_{n-1}^2))^{p'}\right)^{1/p'}}{1}\\
&=\frac{N^{-(n-1)}\left(\sum_{k_1=1}^N\dots \sum_{k_{n-1}=1}^N  1\right)^{1/p'}}{1}=N^{-\frac{n-1}{p}}.
\end{align*}
Observe that
$
-\frac{n-1}{p}\le-(n+1)(\frac{2}p-1)
$
if and only if $p\ge \frac{n+3}{n+1}$, which concludes the proof of the necessity part of the theorem.

 \end{proof}


\begin{thebibliography}{99}

\bibitem{AM}
\textsc{T{.} C{.} Anderson, J{.} Madrid}
\newblock{New bounds for discrete lacunary spherical averages.}
\newblock {\textit{Preprint} (2020).}

\bibitem{bourg}
\textsc{J{.} Bourgain}
\newblock{Fourier transform restriction phenomena for certain lattice subsets and applications to nonlinear evolution equations. I. Schr\"odinger equations.}
\newblock { Geom. Funct. Anal. 3 (1993), no. 2, 107--156. }

\bibitem{BD}
\textsc{J{.} Bourgain, {C.} Demeter}
\newblock{The proof of the $\ell^2$ decoupling conjecture.}
\newblock {Ann. of Math. (2) 182 (2015), no. 1, 351--389. }

\bibitem{HLY}
\textsc{R{.} Han, M{.} T{.} Lacey, F{.} Yang}
\newblock{Averages along the Square Integers: $\ell^p$ improving and Sparse Inequalities}
\newblock {\textit{Preprint} (2019).}

\bibitem{HKLY}
\textsc{R{.} Han, B{.} Krause, M{.} T{.} Lacey, F{.} Yang}
\newblock{Averages along the Primes: Improving and sparse bounds}
\newblock {\textit{Preprint} (2019).}

\bibitem{HKLMY}
\textsc{R{.} Han, V{.} Kova\v{c}, M{.} T{.} Lacey, {J.} Madrid, F{.} Yang}
\newblock{$\ell^p$-improving estimates for discrete polynomial averages via bounds for exponential sums.}
\newblock {\textit{Preprint} (2019).}

\bibitem{K} \textsc{R{.} Kesler}
\newblock{$\ell^p(\mathbb{Z}^d)$-Improving Properties and Sparse Bounds for Discrete Spherical Maximal Means, Revisited.}
\newblock {\textit{Preprint} (2018).}

\bibitem{KL} \textsc{R{.} Kesler, M{.} T{.} Lacey}
\newblock{$\ell^p$-improving inequalities for Discrete Spherical Averages.}
\newblock {\textit{Preprint} (2018).}

\bibitem{madrid}
\textsc{J{.} Madrid}
\newblock{A note about the $\ell^p$-improving property of the average operator.}
\newblock {\textit{Preprint} (2019).}

\end{thebibliography}
\end{document}